\documentclass[11pt,a4paper]{article}
\usepackage{amsmath}     				
\usepackage{amssymb}   					
\usepackage{amsthm}     				
\usepackage{graphicx}    				

\usepackage{textcomp}    				
\usepackage[T1]{fontenc} 				
\usepackage{amsmath}     				
\usepackage{amssymb}   					
\usepackage{amsthm}     				
\usepackage{graphicx}    				

\usepackage{textcomp}    				
\usepackage[T1]{fontenc} 				
\usepackage{marvosym}    				
\usepackage{authblk} 					
\usepackage[usenames]{xcolor}  			
\usepackage{lastpage} 					

\usepackage{enumerate}	 				
\usepackage{mathtools}					
\usepackage{bm}						
\usepackage[noadjust]{cite}			
\usepackage[utf8]{inputenc}
\usepackage[unicode]{hyperref}

\newtheorem{theorem}{Theorem}[section]

\newtheorem{proposition}[theorem]{Proposition}

\theoremstyle{definition}

\theoremstyle{definition}

\theoremstyle{definition}

\numberwithin{equation}{section}
\numberwithin{table}{section}
\numberwithin{figure}{section}

\advance\textheight by 1truein

\newcommand{\ep}{\varepsilon}

\begin{document}

\title{Global stability of a  multistrain SIS model\\ with superinfection and  patch structure}
\author[1]{Attila D\'enes}

\author[2]{Yoshiaki Muroya}

\author[1,3]{Gergely R\"ost}


\affil[1]{{Bolyai Institute}, {University of Szeged}, {{Hungary}}}

\affil[2]{{Department of Mathematics}, {Waseda University}, {{Japan}}}

\affil[3]{{Mathematical Institute}, {University of Oxford}, {{United Kingdom}}}
\maketitle

\abstract{We study the global stability of a  multistrain SIS model with superinfection and  patch structure.  
We establish an iterative procedure to obtain a sequence of threshold parameters.  By a repeated application of a result by Takeuchi et al.\ [\textit{Nonlinear Anal Real World Appl}.\ 2006;{7}:235--247], we   show that these parameters completely determine the global dynamics of the system: for any number of patches and strains with different infectivities,  any subset of the strains can stably coexist depending on the particular  choice of the parameters. Finally, we  return to the special case of one patch examined in [\textit{Math Biosci Eng}.\ 2017;14:421--435] and give a correction to the proof 
of Theorem 2.2 of that paper. }



\section{Introduction}
Several  viruses have  different genetic variants (subtypes)  called strains which may differ in their infectivity and  virulence. Stronger strains might superinfect an individual already infected by another strain and there can be a coexistence of different virus strains with different virulence. Nowak \cite{Nowak} considered a model to provide an analytical understanding of the complexities introduced by superinfection. 
In our earlier work\cite{DMR}, we considered a multistrain SIS model with superinfection with $n$ infectious strains and showed that it is possible to obtain a stable coexistence of any subgroup of the $n$ strains. We established an iterative method for calculating a sequence of reproduction numbers, which determine the strains being  present in the globally asymptotically stable coexistence equilibrium. 

Recently, there has been an increasing interest in the modelling of the spatial spread of infectious diseases (see e.g.\ Arino and Portet \cite{arino}, Knipl~\cite{Knipl:2014}, Knipl and R\"ost \cite{Knipl:2016}, Muroya, Kuniya and Enatsu \cite{Muroya:2016}, Nakata and R\"ost \cite{Nakata}). There are several ways to model spatial spread: one might use partial differential equations (see e.g.\ Peng and Zhao \cite{pengzhao}, Allen et al.~\cite{allen}, Ge et al.~\cite{huaiping}) or one may apply ordinary or functional differential equations  where individuals can travel between different patches (countries, regions, cities etc.).

 Marv\'a et al.\ \cite{Marva:2012} considered a spatially distributed periodic multistrain SIS epidemic model with patches of periodic migration rates without superinfection. Considering global reproduction numbers in the non-spatialized aggregated
system that serve to decide the eradication or endemicity of the epidemic in the initial spatially distributed nonautonomous model, and comparing these global reproductive numbers with those corresponding to isolated patches, they showed that adequate periodic fast migrations can in many cases reverse local endemicity and get global eradication of the epidemic.

Motivated by our earlier work on multistrain models and by the recent results on spatial spread of diseases, we extend our previous model \cite{DMR} to the general case of $p$ patches. In Section 2, we establish a  multistrain SIS model with superinfection with $n$ infectious strains and patch structure.  In Section 3, we establish an iterative procedure to determine the globally asymptotically stable equilibrium of  the multipatch model introduced in Section 2. In Section 4, we turn to the case $p=1$, studied in D\'enes, Muroya and R\"ost \cite{DMR} and give a correction of the proof of Theorem 2.2 of that paper. 

\section{The model}\label{sec2}
We consider a heterogeneous virus population with $n$ virus strains having  different infectivities and virulences.   We will assume that superinfection is possible, and more virulent strains
outcompete the less virulent ones in an infected individual taking over the host completely, i.e.\ we assume that an infected individual is always infected by only one virus strain. Let $n$ denote the number of strains with different virulences while $p$ stands for the number of patches. On each patch, the population is divided into $n + 1$ compartments depending on the presence of any of the virus strains: the susceptible class of patch $\ell$ is denoted by $S^\ell(t)$ and on each patch $\ell$, there are $n$ infected compartments $T^\ell_1,\dots,T^\ell_n$ where a larger index corresponds to a compartment of individuals infected by a strain with larger virulence, so for $i<j$,  $T_j$ individuals superinfect $T_i$ individuals. Let $B^\ell$ denote the birth rate and $b^\ell$ the death rate on the $\ell$th patch.
We denote by  $\beta^\ell_{kj}$  the transmission rate on patch $\ell$ by which the $k$th strain infects those who are infected by the $j$th strain. The transmission rates from susceptibles to
strain $k$ on patch $\ell$ will be denoted by $\beta^\ell_{kk}$.  Recovery rate on patch $\ell$ among those infected by
the $k$th strain will be denoted by $\theta_k^\ell$.  By $m_{\ell i}$ we denote the travel rate from patch $i$ to $\ell$, which, on a given  patch is equal for all compartments on that patch.
 Using these notations, we consider the following  multistrain  SIS model with superinfection and patch structure:
\begin{equation}\label{eq-n}
\begin{split}
\frac{\mathrm{d}S^\ell(t)}{\mathrm{d}t}={}&B^\ell-b^\ell S^\ell(t)- S^\ell(t) \sum_{k=1}^n \beta_{kk}^\ell T_k^\ell(t)+\sum_{k=1}^n \theta_k^\ell T_k^\ell(t) +\sum_{i=1}^p(1-\delta_{\ell i})\left\{m_{\ell i} S^i(t)-m_{i\ell} S^\ell(t)\right\}, \\
\frac{\mathrm{d}T_k^\ell(t)}{\mathrm{d}t}={}& S^\ell(t) \beta_{kk}^\ell T_k^\ell(t)+T_k^\ell(t) \sum_{j=1}^n (1-\delta_{kj}) \beta_{kj}^\ell T_j^\ell(t)-\left(b^\ell+\theta_k^\ell\right) T_k^\ell(t)+\sum_{i=1}^p (1-\delta_{\ell i})\left\{m_{\ell i} T_k^i(t)-m_{i\ell} T_k^\ell(t)\right\},
\\
& k=1,2,\dots,n, \quad \ell=1,2,\ldots,p,
\end{split}
\end{equation}
with initial conditions
\begin{equation}\label{init}
\begin{split}
&\displaystyle{S^\ell\left( 0\right) =\phi_0^\ell, \qquad T_k^\ell\left( 0\right) =\phi_k^\ell, \qquad k=1,2,\dots,n, \quad \ell=1,2,\ldots,p}, \\
&\displaystyle{\left(\phi_0^1,\phi_1^1,\phi_2^1,\dots,\phi_n^1,\phi_0^2,\phi_1^2,\phi_2^2,\dots,\phi_n^2,\dots, \phi_0^p,\phi_1^p,\phi_2^p,\dots,\phi_n^p \right) \in {\mathbb R}_{+}^{(n+1)p}}\eqqcolon\Gamma,
\end{split}
\end{equation}
where $\delta_{kj}$ denotes the Kronecker delta such that $\delta_{kj}=1$ if $k=j$ and $\delta_{kj}=0$ otherwise, and
where
\begin{equation}\label{cond-0}
\begin{aligned}
\beta_{kj}^\ell&=\beta_{kk}^\ell, &\quad 1 &\leq j \leq k, \quad
 \mbox{and} \\
  \beta_{kj}^\ell&=-\beta_{jj}^\ell, &\quad k+1 &\leq j \leq n, \quad k=1,2,\ldots,n, \quad \ell=1,2,\ldots,p.
\end{aligned}
\end{equation}

Note that for $n=2$ and $p=1$, \eqref{eq-n} corresponds to the model by A.~D\'enes and G.~R\"ost describing the spread of ectoparasites and ectoparasite-borne diseases  \cite{Denes:2012,Denes:2014a}, while for $p=1$, it corresponds to the multistrain SIS model by A.~D\'enes, Y.~Muroya and G.~R\"ost~\cite{DMR}.

\section{Main result}\label{secmain}

Let us introduce the notation 
\begin{equation}\label{Nk}
N_n^\ell(t)=S^\ell(t)+\sum_{j=1}^n T_j^\ell(t), \qquad  \ell=1,2,\ldots,p.
\end{equation}
Then, by \eqref{cond-0}, we have $\beta_{kj}^\ell=-\beta_{jk}^\ell$ for $k \neq j$ and hence, $$\displaystyle{\sum_{k=1}^n T_k^\ell(t) \sum_{j=1}^n (1-\delta_{kj}) \beta_{kj}^\ell T_j^\ell(t)=0, \qquad \ell=1,2,\ldots,p}.$$ 
Thus, \eqref{eq-n} is equivalent to
\allowdisplaybreaks
\begin{subequations}\label{eq-n-1}
\begin{align}
\begin{split}\label{eq-n-1-1}
\frac{\mathrm{d}T_k^\ell(t)}{\mathrm{d}t}={}&\biggl(N_n^\ell(t)-\sum_{j=1}^n T_j^\ell(t)\biggr) \beta_{kk}^\ell T_k^\ell(t)+T_k^\ell(t)\sum_{j=1}^n (1-\delta_{kj}) \beta_{kj}^\ell T_j^\ell(t) -\left(b^\ell+\theta_k^\ell\right) T_k^\ell(t) 
\\&+\sum_{i=1}^p(1-\delta_{\ell i})\left\{m_{\ell i} T_k^i(t)-m_{i\ell} T_k^\ell(t)\right\}, \qquad k=1,2,\ldots,n-1, 
\end{split}\\
\begin{split}\label{eq-n-1_Tn}
\frac{\mathrm{d}T_n^\ell(t)}{\mathrm{d}t}={}&\biggl(N_n^\ell(t)-\sum_{j=1}^n T_j^\ell(t)\biggr) \beta_{nn}^\ell T_n^\ell(t)+T_n^\ell(t)\sum_{j=1}^n (1-\delta_{nj}) \beta_{nj}^\ell T_j^\ell(t) \\
&-\left(b^\ell+\theta_n^\ell\right) T_n^\ell(t)+\sum_{i=1}^p(1-\delta_{\ell i})\left\{m_{\ell i} T_n^i(t)-m_{i\ell} T_n^\ell(t)\right\} \\
={}&T_n^\ell(t)\biggl(\beta_{nn}^\ell N_n^\ell(t)-\sum_{j=1}^n \left\{\beta_{nn}^\ell-(1-\delta_{nj})\beta_{nj}^\ell\right\} T_j^\ell(t)-\left(b^\ell+\theta_n^\ell\right)\biggr) +\sum_{i=1}^p(1-\delta_{\ell i})\left\{m_{\ell i} T_n^i(t)-m_{i\ell} T_n^\ell(t)\right\}, \\
={}& T_n^\ell(t)\biggl(\beta_{nn}^\ell N_n^\ell(t)- \beta_{nn}^\ell T_n^\ell(t)-\left(b^\ell+\theta_n^\ell\right)\biggr) +\sum_{i=1}^p(1-\delta_{\ell i})\left\{m_{\ell i} T_n^i(t)-m_{i\ell} T_n^\ell(t)\right\}, 
\end{split}\\
\begin{split}\label{eq-n-1_N}
\frac{\mathrm{d}N_n^\ell(t)}{\mathrm{d}t}={}&B^\ell-b^\ell N_n^\ell(t)+\sum_{i=1}^p(1-\delta_{\ell i})\left\{m_{\ell i} N_n^i(t)-m_{i\ell} N_n^\ell(t)\right\},\\
& \ell=1,2,\ldots,p.
\end{split}
\end{align}
\end{subequations}
The equations \eqref{eq-n-1_Tn}--\eqref{eq-n-1_N} are clearly independent from the rest of the equations. In particular, the equations \eqref{eq-n-1_N} are also independent from the equations \eqref{eq-n-1_Tn}. As the coefficient matrix $A$ of the linear system of equations 
$$
\begin{pmatrix}B^1\\\vdots\\B^p
\end{pmatrix}=
\begin{pmatrix}
b^1+\sum_{i=1}^{p}(1-\delta_{1i})m_{i1}&-m_{12}&\cdots&-m_{1p}\\
-m_{21}&b^2+\sum_{i=1}^{p}(1-\delta_{2i})m_{i2}&\cdots&-m_{2p}\\
\vdots&\vdots&\ddots&\vdots\\
-m_{p1}&-m_{p2}&\cdots& b^p+\sum_{i=1}^{p}(1-\delta_{pi})m_{ip}
\end{pmatrix}\begin{pmatrix}
N_n^{1}\\ \vdots\\N_n^{p}
\end{pmatrix}$$
is a strictly diagonally dominant $Z$-matrix, it is nonsingular and its inverse is positive, hence, this algebraic system has a unique, positive solution
$$
\begin{pmatrix}
N_n^{1*}\\ \vdots\\N_n^{p*}
\end{pmatrix}=A^{-1}\begin{pmatrix}B^1\\\vdots\\B^p
\end{pmatrix}.$$
Let us define $P_\ell(t)\coloneqq N^\ell(t)-N^{\ell*},\ \ell=1,\dots,p$, then for $P_\ell'(t)$, we have the equation
\begin{equation}\label{linP}
\frac{\mathrm{d}}{\mathrm{d}t}\begin{pmatrix}
P_1(t)\\ \vdots\\P_p(t)
\end{pmatrix}=-A\begin{pmatrix}
P_1(t)\\ \vdots\\P_p(t)
\end{pmatrix}.
\end{equation}
From the properties of the matrix $-A$, applying the Gershgorin circle theorem, we obtain that $P_\ell(t)\to 0$ exponentially as $t\to\infty, \ \ell=1,\dots,p$.
Hence, for the equations \eqref{eq-n-1_N}, there exist positive constants $N_n^{\ell*}, \ \ell=1,2,\ldots,p$ such that   
\begin{equation}\label{limit1}
\lim_{t \to +\infty}N_n^\ell(t)=N_n^{\ell*}, \qquad \ell=1,2,\ldots,p,
\end{equation}
exponentially and  \eqref{eq-n-1_Tn} has the following limit  system:
\begin{equation}\label{eq-n-2_2}
\begin{split}
\frac{\mathrm{d}T_n^\ell(t)}{\mathrm{d}t}={}&T_n^\ell(t)\left(\beta_{nn}^\ell N_n^{\ell*}-\left(b^\ell+\theta_n^\ell\right)-\beta_{nn}^{\ell}T_n^{\ell}(t)\right) 
+\sum_{i=1}^p(1-\delta_{\ell i})\left\{m_{\ell i} T_n^i(t)-m_{i\ell} T_n^\ell(t)\right\}, \qquad \ell =1,2,\ldots,p,
\end{split}
\end{equation}
which is a $p$-dimensional Lotka--Volterra system with patch structure, in the form as Equation (2.1) in Takeuchi et al.~\cite{Takeuchi:2006}

We introduce the notation
\begin{align*}
\tilde{m}_{ii}&=\sum_{\ell=1}^p(1-\delta_{i\ell})m_{i\ell}, \qquad i=1,2,\ldots,p, 
\end{align*}
and define the connectivity matrix
\begin{align*}
M&=
\begin{bmatrix}
	-\tilde{m}_{11}&m_{12}&\cdots&m_{1p} \\
	m_{21}&-\tilde{m}_{22}&\cdots&m_{2p} \\
	\vdots&\vdots&\ddots&\vdots \\
	m_{p1}&m_{p2}&\cdots&-\tilde{m}_{pp} \\
\end{bmatrix}. 
\end{align*}

Now we define
\begin{align*}
c_n^\ell&=\beta_{nn}^\ell N_n^{\ell*}-(b^\ell+\theta_n^\ell), \qquad \ell=1,2,\ldots,p,
\shortintertext{and}
M_n&=
\begin{bmatrix}
c_n^1-\tilde{m}_{11}&m_{12}&\cdots&m_{1p} \\
m_{21}&c_n^2-\tilde{m}_{22}&\cdots&m_{2p} \\
\vdots&\vdots&\ddots&\vdots \\
m_{p1}&m_{p2}&\cdots&c_n^p-\tilde{m}_{pp} \\
\end{bmatrix}. 
\end{align*}
Let us denote by $s(L)$ the stability modulus of a $p \times p$ matrix $L$, defined by
$s(L) \coloneqq \max\{\operatorname{Re} \lambda : \lambda  \mbox{ is an eigenvalue} \mbox{ of} \ L \}$.
If $L$ has nonnegative off-diagonal elements and is irreducible, then $s(L)$ is a simple eigenvalue
of $L$ with a (componentwise) positive eigenvector (see, e.g., Theorem A.5 in Smith~\cite{Smith:1995}).

\begin{proposition}[see Theorem 2.1 in Takeuchi et al.~{\cite{Takeuchi:2006}}] \label{pro21}
Suppose that $M_n$ is irreducible. Then  equation \eqref{eq-n-2_2}
 has a positive equilibrium 
which is globally asymptotically stable if $s(M_n)>0$.  If $s(M_n) \leq 0$, then $0$ is a globally asymptotically stable equilibrium and the populations go extinct
in every patch.
\end{proposition}
Note that we may take that the populations go extinct in every patch  not only if $s(M_n)<0$ but also if $s(M_n)=0$ (see Theorem~2.2 of Faria \cite{Faria:2013}). 

Let $E_n^*=(T_n^{1*},T_n^{2*},\ldots,T_n^{p*})$ be the unique equilibrium of \eqref{eq-n-2_2} which is globally asymptotically stable. Then, $E_n^*=(0,0,\ldots,0)$ if $s(M_n) \leq 0$, and $E_n^*=(T_n^{1*},T_n^{2*},\ldots,T_n^{p*})$ satisfies $T_n^{\ell*}>0, \ \ell=1,2,\ldots,p$, if $s(M_n)>0$. Therefore, in the first case, the unique equilibrium of \eqref{eq-n-2_2},  is globally asymptotically stable on $\{(T_n^1,T_n^2,\ldots,T_n^p) \in \mathbb{R}_+^p \}$, while in the second case, the unique positive equilibrium $E_n^*=(T_n^{1*},T_n^{2*},\ldots,T_n^{p*})$ with $T_n^{\ell*}>0, \ \ell=1,2,\ldots,p$ is globally asymptotically stable with respect to $\{(T_n^1,T_n^2,\ldots,T_n^p) \in \mathbb{R}_+^p \}\setminus \{(0,0,\ldots,0)\}$. Let us introduce the notations 
\begin{align*}
N_{n-1}^\ell(t)&=S^\ell(t)+\sum_{j=1}^{n-1} T_j^\ell(t), \qquad  \ell=1,2,\ldots,p,
\shortintertext{and} 
		b^\ell_{(1)}&=b^{\ell}-\beta_{kn}^\ell T_n^{\ell*}=b^{\ell}+\beta_{nn}^\ell T_n^{\ell*}, \qquad k=1,2,\ldots,n-1,\quad \ell=1,2,\ldots,p,
		 \shortintertext{and}  
		  B^\ell_{(1)}&=B^{\ell}+\theta_{n}^\ell T_n^{\ell*}, \qquad \ell=1,2,\ldots,p,
\end{align*}
where 
$(T_n^{1*},\dots,T_n^{1*})$ is either equal to $(0,\dots,0)$ (if $s(M_n) \leq 0$) or it is equal to the unique positive equilibrium of \eqref{eq-n-2_2} (if $s(M_n) > 0$). This way, substituting $T_n^{i*},\ 1=1,\dots,p$ into the place of $T_n^i(t)$ in \eqref{Nk} and \eqref{eq-n-1},
we may consider the following reduced system of \eqref{eq-n-1} for the global stability of \eqref{eq-n}:  
\allowdisplaybreaks
\begin{subequations}\label{eq-(n-1)-1}
\begin{align}
\begin{split}
\frac{\mathrm{d}T_k^\ell(t)}{\mathrm{d}t}={}&\biggl(N_{n-1}^\ell(t)-\sum_{j=1}^{n-1} T_j^\ell(t)\biggr) \beta_{kk}^\ell T_k^\ell(t)+T_k^\ell(t)\sum_{j=1}^{n-1} (1-\delta_{kj}) \beta_{kj}^\ell T_j^\ell(t)-\left(b^\ell_{(1)}+\theta_k^{\ell}\right) T_k^\ell(t) \\&+\sum_{i=1}^p(1-\delta_{\ell i})\left\{m_{\ell i} T_k^i(t)-m_{i\ell} T_k^\ell(t)\right\}, \qquad k=1,2,\ldots,n-2,
\end{split}\\
\begin{split}
\frac{\mathrm{d}T_{n-1}^\ell(t)}{\mathrm{d}t}={}&\biggl(N_{n-1}^\ell(t)-\sum_{j=1}^{n-1} T_j^\ell(t)\biggr) \beta_{n-1,n-1}^\ell T_{n-1}^\ell(t) +T_{n-1}^\ell(t)\sum_{j=1}^{n-1} (1-\delta_{n-1,j}) \beta_{n-1,j}^\ell T_j^\ell(t)-\left(b^\ell_{(1)}+\theta_{n-1}^{\ell}\right) T_{n-1}^\ell(t) \\
& +\sum_{i=1}^p (1-\delta_{\ell i})\left\{m_{\ell i} T_{n-1}^i(t)-m_{i\ell} T_{n-1}^\ell(t)\right\}\\
={}&T_{n-1}^\ell(t)\biggl(\beta_{n-1,n-1}^\ell N_{n-1}^\ell(t)  -\beta_{n-1,n-1}^\ell T_{n-1}^\ell(t)-\left(b^\ell_{(1)}+\theta_{n-1}^{\ell}\right)\biggr) \\
 &+\sum_{i=1}^p(1-\delta_{\ell i})\left\{m_{\ell i} T_{n-1}^i(t)-m_{i\ell} T_{n-1}^\ell(t)\right\}, 
 \end{split}\\
 \begin{split}\label{eq-(n-1)-1-N}
 \frac{\mathrm{d}N_{n-1}^\ell(t)}{\mathrm{d}t}={}&B^\ell_{(1)}-b^\ell_{(1)} N_{n-1}^\ell(t)+\sum_{i=1}^p(1-\delta_{\ell i})\left\{m_{\ell i} N_{n-1}^i(t)-m_{i\ell} N_{n-1}^\ell(t)\right\}, \\
&\ell=1,2,\ldots,p.
\end{split}
\end{align}
\end{subequations}
It is easy to see that  \eqref{eq-(n-1)-1}  is of similar structure as \eqref{eq-n-1}, but with dimension $p(n-1)+1$. The positivity of the new parameters follows from the conditions \eqref{cond-0}. This means that by repeating the above steps, namely, substituting the limit of the total populations in the patches and then substituting the limit of the Lotka--Volterra system for the strongest strain, we  can further reduce the dimension by substituting the values of the equilibrium which is globally asymptotically stable, of the decoupled $p$ dimensional Lotka--Volterra system into the remaining equations. 

In general, after performing the above steps $q$ times, we arrive at the system
\allowdisplaybreaks
\begin{subequations}\label{eq-(n-q)-1}
\begin{align}\label{18a}
\begin{split}
\frac{\mathrm{d}T_k^\ell(t)}{\mathrm{d}t}={}&\biggl(N_{n-q}^\ell(t)-\sum_{j=1}^{n-q}T_j^\ell(t)\biggr) \beta_{kk}^\ell T_k^\ell(t)+T_k^\ell(t)\sum_{j=1}^{n-q} (1-\delta_{kj}) \beta_{kj}^\ell T_j^\ell(t)\\
&-\left(b^\ell_{(q)}+\theta_k^{\ell}\right) T_k^\ell(t)  +\sum_{i=1}^p (1-\delta_{\ell i})\left\{m_{\ell i} T_k^i(t)-m_{i\ell } T_k^\ell(t)\right\}, \\
 &k=1,2,\ldots,n-q-1, 
\end{split}\\
\begin{split}\label{2.15b}
\frac{\mathrm{d}T_{n-q}^\ell(t)}{\mathrm{d}t}={}&\biggl(N_{n-q}^\ell(t)-\sum_{j=1}^{n-q} T_j^\ell(t)\biggr) \beta_{n-q,n-q}^\ell T_{n-q}^\ell(t) \displaystyle{+T_{n-q}^\ell(t)\sum_{j=1}^{n-q} (1-\delta_{n-q,j}) \beta_{n-q,j}^\ell T_j^\ell(t)-\left(b^\ell_{(q)}+\theta_{n-q}^{\ell}\right) T_{n-q}^\ell(t)} \\
& \displaystyle{+\sum_{i=1}^p (1-\delta_{\ell i})\left\{m_{\ell i} T_{n-q}^i(t)-m_{i\ell} T_{n-q}^\ell(t)\right\}}, \\
={}&T_{n-q}^\ell(t)\biggl(\beta_{n-q,n-q}^\ell N_{n-q}^\ell(t)  -\beta_{n-q,n-q}^\ell T_{n-q}^\ell(t)-\left(b^\ell_{(q)}+\theta_{n-q}^{\ell}\right)\biggr) \\
&+\sum_{i=1}^p (1-\delta_{\ell i})\left\{m_{\ell i} T_{n-q}^i(t)-m_{i\ell} T_{n-q}^\ell(t)\right\}, 
\end{split}\\
\begin{split}\label{eq-(n-q)-1-N}
\frac{\mathrm{d}N_{n-q}^\ell(t)}{\mathrm{d}t}={}&B^\ell_{(q)}-b^\ell_{(q)} N_{n-q}^\ell(t)+\sum_{i=1}^p (1-\delta_{\ell i})\left\{m_{\ell i} N_{n-q}^i(t)-m_{i\ell} N_{n-q}^\ell(t)\right\}, \\
&  \ell=1,2,\ldots,p,
\end{split}
\end{align}
\end{subequations}
where
$$N_{n-q}^\ell(t)=S^\ell(t)+\sum_{j=1}^{n-q}T_j^\ell(t),\qquad \ell=1,2,\ldots,p,$$ and
\begin{equation*}
\begin{split}
b^\ell_{(q)}&=b^{\ell}_{(q-1)}-\beta_{k,{n-q+1}}^\ell T_{n-q+1}^{\ell*}=b^{\ell}_{(q-1)}+\beta_{n-q+1,{n-q+1}}^\ell T_{n-q+1}^{\ell*}, \qquad k=1,2,\ldots,n-q,\quad \ell=1,2,\ldots,p\\
B^\ell_{(q)}&=B^{\ell}_{(q-1)}+\theta_{n-q+1}^\ell T_{n-q+1}^{\ell*}, \qquad  \ell=1,2,\ldots,p. 
\end{split}
\end{equation*}
From the  equations  \eqref{eq-(n-q)-1-N}, similarly as before, there exist positive constants $N_{n-q}^{\ell*}, \ \ell=1,2,\ldots,p$ 
 such that   
\begin{equation}\label{limitk}
\lim_{t \to +\infty}N_{n-q}^\ell(t)=N_{n-q}^{\ell*}, \qquad \ell=1,2,\ldots,p,
\end{equation}
and  \eqref{eq-(n-q)-1} has the following reduced limit system:	
\begin{subequations}\label{eq-(n-q)-2_1}
\begin{align}\label{eq-(n-q)-2}
\begin{split}
\frac{\mathrm{d}T_k^\ell(t)}{\mathrm{d}t}={}&\biggl( N_{n-q}^\ell(t)-\sum_{j=1}^{n-q}T_j^\ell(t)\biggr) \beta_{kk}^\ell T_k^\ell(t)+T_k^\ell(t)\sum_{j=1}^{n-q} (1-\delta_{kj}) \beta_{kj}^\ell T_j^\ell(t)\\
&-\left(b^\ell_{(q)}+\theta_k^{\ell}\right) T_k^\ell(t) +\sum_{i=1}^p(1-\delta_{\ell i})\left\{m_{\ell i} T_k^i(t)-m_{i\ell} T_k^\ell(t)\right\}, \\ &k=1,2,\ldots,n-q-1, 
\end{split}\\
\begin{split}\label{2.17b}
\frac{\mathrm{d}T_{n-q}^\ell(t)}{\mathrm{d}t}={}&	T_{n-q}^\ell(t)\biggl(\beta_{n-q,n-q}^\ell N_{n-q}^{\ell*}-(b^\ell_{(q)}+\theta_{n-q}^{\ell}) -\beta_{n-q,n-q}^\ell T_{n-q}^\ell(t)\biggr) +\sum_{i=1}^p (1-\delta_{\ell i})\left\{m_{\ell i} T_{n-q}^i(t)-m_{i\ell} T_{n-q}^\ell(t)\right\}, \\& \ell=1,2,\ldots,p.
\end{split}
\end{align}
\end{subequations}
%

Let us define
\begin{align*}
M_{n-q}&=
\begin{bmatrix}
c_{n-q}^1-\tilde{m}_{11}&m_{12}&\cdots&m_{1{p}} \\
m_{21}&c_{n-q}^2-\tilde{m}_{22}&\cdots&m_{2p} \\
\vdots&\vdots&\ddots&\vdots \\
m_{p1}&m_{p2}&\cdots&c_{n-q}^p-\tilde{m}_{pp} \\
\end{bmatrix}, \\
\shortintertext{with}
c_{n-q}^\ell&=\beta_{n-q,n-q}^\ell N_{n-q}^{\ell*}-(b^\ell_{(q)}+\theta_{n-q}^{\ell}), \qquad \ell=1,2,\ldots,p.
\end{align*}
Again, \eqref{2.17b} can be  the decoupled from the rest of the equations as a $p$ dimensional Lotka--Volterra system with patch structure:
\begin{equation}\label{eq-(n-q)-2-new}
\begin{split}
\frac{\mathrm{d}T_{n-q}^\ell(t)}{\mathrm{d}t}={}&T_{n-q}^\ell(t)\biggl(\beta_{n-q,n-q}^\ell N_{n-q}^{\ell*}-\left(b^\ell_{(q)}+\theta_{n-q}^{\ell}\right) -\beta_{n-q,n-q}^\ell T_{n-q}^\ell(t)\biggr) \\
&+\sum_{i=1}^p (1-\delta_{\ell i})\left\{m_{\ell i} T_{n-q}^i(t)-m_{i\ell} T_{n-q}^\ell(t)\right\}, \qquad \ell=1,2,\ldots,p.
\end{split}
\end{equation}
Similarly as before, assuming the irreducibility of $M_{n-q}$, this system  has a globally attractive  equilibrium $(T_{n-q}^{1*},T_{n-q}^{2*},\ldots,T_{n-q}^{p*})$, which is either the trivial equilibrium if $s(M_{n-q})\leq 0$ or a positive equilibrium if $s(M_{n-q})>0$. 

Let us now define the new coefficients
\begin{align*}
b^\ell_{(q+1)}&=b^{\ell}_{(q)}-\beta_{k,{n-q}}^\ell T_{n-q}^{\ell*}=b^{\ell}_{(q)}+\beta_{{n-q},{n-q}}^\ell T_{n-q}^{\ell*}, \qquad k=1,2,\ldots,n-q-1,\quad \ell=1,2,\ldots,p \\
 \shortintertext{and}
 B^\ell_{(q+1)}&=B^{\ell}_{(q)}+\theta_{k}^\ell T_{n-q}^{\ell*}, \qquad k=1,2,\ldots,n-q-1,\quad \ell=1,2,\ldots,p. 
\end{align*}
and the new variables
$$N_{n-q-1}^\ell(t)=S^\ell(t)+\sum_{j=1}^{n-q-1}T_j^\ell(t),\qquad \ell=1,2,\ldots,p.$$
 We obtain the system
\begin{subequations}\label{eq-(n-q-1)-2}
\begin{align}
\begin{split}
\frac{\mathrm{d}T_k^\ell(t)}{\mathrm{d}t}={}&\biggl(N_{n-q-1}^\ell(t)-\sum_{j=1}^{n-q-1}T_j^\ell(t)\biggr) \beta_{kk}^\ell T_k^\ell(t)+T_k^\ell(t)\sum_{j=1}^{n-q-1} (1-\delta_{kj}) \beta_{kj}^\ell T_j^\ell(t)-\left(b^\ell_{(q+1)}+\theta_k^{\ell}\right) T_k^\ell(t)\\
& +\sum_{i=1}^p(1-\delta_{\ell i})\left\{m_{\ell i} T_k^i(t)-m_{i\ell} T_k^\ell(t)\right\}, \\
& k=1,2,\ldots,n-q-2, \quad \ell=1,2,\ldots,p,
\end{split}\\
\begin{split}
\frac{\mathrm{d}T_{n-q-1}^\ell(t)}{\mathrm{d}t}={}&T_{n-q-1}^\ell(t)\biggl(\beta_{n-q-1,n-q-1}^\ell N_{n-q-1}^{\ell}(t)-\left(b^\ell_{(q+1)}+\theta_{n-q-1}^{\ell}\right) -\beta_{n-q-1,n-q-1}^\ell T_{n-q-1}^\ell(t)\biggr) \\
&+\sum_{i=1}^p (1-\delta_{\ell i})\left\{m_{\ell i} T_{n-q-1}^i(t)-m_{i\ell} T_{n-q-1}^\ell(t)\right\}, \\& \ell=1,2,\ldots,p,
\end{split}\\
\begin{split}
\frac{\mathrm{d}N_{n-q-1}^\ell(t)}{\mathrm{d}t}={}&B_{(q+1)}^\ell-b_{(q+1)}^\ell N_{n-q-1}^\ell(t)+\sum_{i=1}^{p}(1-\delta_{\ell i})\left\{m_{\ell i}N_{n-q-1}^i(t)-m_{i\ell}N_{n-q-1}^\ell(t)\right\},\\
&\ell=1,2,\dots,p,
\end{split}
\end{align}
\end{subequations}
which again, is a system with the same structure. In the end, we arrive at a $p$ dimensional Lotka--Volterra system, the dynamics of which can be determined in a similar way as in the above case. This final system will give us an equilibrium value for $S^1(t)$ and $(T_1^1(t),T_2^1(t),\ldots,T_p^1(t))$. Thus, by the above discussion, we can reach a conclusion by induction to the global dynamics of the model \eqref{eq-n} and we formulate the following theorem.

\begin{theorem}\label{thm22}
Assume that the connectivity matrix $M$ is irreducible. Then the global dynamics of the multistrain, multipatch SIS model  \eqref{eq-n} is completely determined by the threshold parameters $(s(M_1),s(M_2),\ldots,s(M_n))$ which can be obtained iteratively. There exists an equilibrium in $\Gamma$ which is globally asymptotically stable with respect to the region $\Gamma_0$, where $\Gamma_0$ is the interior of~$\Gamma$.
\end{theorem}

\begin{proof}[Proof of Theorem \ref{thm22}]

The main part of the proof consists of the above description of the steps of the procedure.
There is one point left to be shown: we have to prove that in each step, when we substitute the limits $N_k^{\ell*}$, resp.\ $T_k^{\ell*}$ into the equations, the dynamics of the resulting system is indeed equivalent to that of the preceding one. 

We summarize the steps of the procedure in the following.
\begin{enumerate}
\item We obtain $N_n^{\ell\ast}\  (\ell=1,\dots,p)$  from the linear system  \eqref{linP}.
\item We substitute the limits $N_n^{\ell\ast}\  (\ell=1,\dots,p)$    into the equations \eqref{eq-n-1_Tn} to obtain the equations \eqref{eq-n-2_2}. 
\item We obtain the limits $T_{n}^{\ell*}\  (\ell=1,\dots,p)$ of the Lotka--Volterra system \eqref{eq-n-2_2}. 
\item We 
create the new variables ${N}_{n-1}^{\ell}(t),\ \ell=1,\dots,p$  and parameters $b_{(1)}^\ell,B_{(1)}^\ell$, $\ell=1,\dots,p$. 
\item We substitute the limits 
$T_{n}^{\ell*}\  (\ell=1,\dots,p)$ into the equations \eqref{eq-n-1-1} to obtain the reduced system \eqref{eq-(n-1)-1} which has the same  structure as the original one \eqref{eq-n-1}. 
\item We repeat this cycle $n-1$ times, with the indices decreased by 1 every time.
\end{enumerate}
For the validity of Step 3 in the $q$th cycle, we need to verify that $M_{n-q}$ is irreducible. Since $M_{n-q}=\!M+\operatorname{diag}[c_{n-q}^1,\dots,c_{n-q}^p]$ and we assumed that $M$ is irreducible, $M_{n-q}$ is also irreducible.

To obtain that in each case, the limit of the solutions of the resulting system after the substitution will be the same equilibrium as the limit of the solutions of the original system, we will apply Theorem 4.1 of Hirsch and Smith \cite{hirsch}. To apply this theorem, we recall the quasimonotone condition \cite{hirsch}  for  a differential equation $x'(t)=f(t,x(t))$: we say that the time-dependent vector field $f\colon J\times D\to\mathbb{R}^n$ (where $J\subset\mathbb{R}$ and $D\subset \mathbb{R}^n$) satisfies the quasimonotone condition in~$D$ if for all $(t,y),(t,z)\in J\times D$,  we have
$$y\leq z \quad\mbox{and} \quad y_i=z_i\quad \mbox{implies}\quad  f_i(t,y)\leq f_i(t,z).$$
According to Theorem 4.1 of Hirsch and Smith \cite{hirsch},  if $f,g\colon J\times D \to \mathbb{R}^n$ are continuous, Lipschitz on each compact subset of $D$,  at least one of them satisfies the quasimonotone condition, and $f(t,y)\leq g(t,y)$ for all $(t,y)\in J\times D$, then $$y,z\in\mathbb{R}^n,\ y\leq z \quad\mbox{implies}\quad x(t;t_0,y)\leq x(t;t_0,z) \quad \mbox{for all } t>t_0,$$
where $x(t;t_0,y)$ denotes the solution of $x'(t)=f(t,x(t))$ started from $y$ at $t=t_0$.


To show that the limits  $T_k^{\ell*}$ obtained during the procedure by substituting the limits of \eqref{eq-(n-q)-2-new} into \eqref{eq-(n-q)-2} are the same as the limit of the variables $T_k^\ell$, $k=1,\dots,n$, $\ell=1,\dots,p$ in the original system, we will	 use an induction argument. It is clear from the above that the claim is true for $k=n$. Let us now suppose that the claim  is not true for all $T_k^\ell(t)$, then there exists a largest index $1\leq r\leq n-1$ such that $T_r^{m*}$ is not equal to the limit of $T_r^m(t)$ in the original system for some $1\leq m\leq p$. The limits $T_r^{\ell*}$ are obtained by first substituting the limits $T_{r+1}^{\ell*}$ into the equations for $T_j^\ell(t)$, $1\leq j\leq r$ and then substituting the limits $N_r^{\ell*}$ into the equations for $T_r^\ell(t)$, hence, we have to compare the limits of the two systems
\begin{align}
\frac{\mathrm{d}T_{r}^\ell(t)}{\mathrm{d}t}={}&\biggl( N_{r+1}^\ell(t)-2T_{r+1}^\ell(t)-T_{r}^\ell(t)\biggr) \beta_{rr}^\ell T_{r}^\ell(t)-\left(b^\ell_{(n-r+1)}+\theta_{r}^{\ell}\right) T_{r}^\ell(t) +\sum_{i=1}^p(1-\delta_{\ell i})\left\{m_{\ell i} T_{r}^i(t)-m_{i\ell} T_{r}^\ell(t)\right\} \nonumber \\
={}&\biggl( N_{r}^\ell(t)-T_{r+1}^\ell(t)-T_{r}^\ell(t)\biggr) \beta_{rr}^\ell T_{r}^\ell(t)-\left(b^\ell_{(n-r+1)}+\theta_{r}^{\ell}\right) T_{r}^\ell(t) +\sum_{i=1}^p(1-\delta_{\ell i})\left\{m_{\ell i} T_{r}^i(t)-m_{i\ell} T_{r}^\ell(t)\right\}\label{15}
\shortintertext{and}
		\frac{\mathrm{d}T_{r}^\ell(t)}{\mathrm{d}t}={}&	\biggl( N_{r}^{\ell*}- T_{r}^\ell(t)\biggr)\beta_{rr}^\ell T_{r}^\ell(t)	-\left(b^\ell_{(n-r)}+\theta_{r}^{\ell}\right)T_{r}^\ell(t)	 +\sum_{\ell=1}^p\label{16} (1-\delta_{\ell i})\left\{m_{\ell i} T_{r}^i(t)-m_{i\ell} T_{r}^\ell(t)\right\}, \\& \ell=1,2,\ldots,p.\nonumber
\end{align}
We know that $N_r^{\ell}(t)$ ($\ell=1,\dots,p$) converge to $N_{r}^{\ell*}$ ($\ell=1,\dots,p$), while from the definition of $r$  we have that   $T_{r+1}^{\ell}(t)$ ($\ell=1,\dots,p$) converge to $T_{r+1}^{\ell*}$ ($\ell=1,\dots,p$). Then, for any $\varepsilon>0$, there exists a $\bar{t}>0$ such  that $|N_{r}^{\ell}(t)-N_{r}^{\ell*}|<\varepsilon$ and  $|T_{r+1}^{\ell}(t)-T_{r+1}^{\ell*}|<\varepsilon$ for all $t>\bar{t}$, $\ell=1,\dots,p$.  If we substitute $T_{r+1}^{1*}+\varepsilon,\dots ,T_{r+1}^{p*}+\varepsilon,N_{r}^{1*}-\varepsilon,\dots,N_{n-q}^{p*}-\varepsilon$, resp.\ $T_{r+1}^{1*}-\varepsilon,\dots ,T_{r+1}^{p*}-\varepsilon,N_{r}^{1*}+\varepsilon,\dots,N_{n-q}^{p*}+\varepsilon$ into \eqref{15}, we obtain two systems of the same structure as \eqref{16}, and one of them is a lower, the other is an upper estimate of  \eqref{15}, and each has a globally asymptotically stable equilibrium $(\underline{T}_{r}^{1}(\varepsilon),\dots,\underline{T}_{r}^{p}(\varepsilon))$, resp.\ $(\overline{T}_{r}^{1}(\varepsilon),\dots,\overline{T}_{r}^{p}(\varepsilon))$ because of Proposition~\ref{pro21}. It is easy to see that the original system \eqref{15}, considered as a nonautonomous system with time-dependent coefficients $T_{r+1}^1(t),\dots,T_{r+1}^p(t),N_{r}^1(t),\dots,N_{r}^p(t)$, satisfies the quasimonotone condition, as well as the systems obtained after the substitution. 
 Hence we can apply Theorem 4.1 of Hirsch and Smith \cite{hirsch} to obtain that for any solution $(T_{r}^{1}(t),\dots,T_{r}^{p}(t))$ of \eqref{15},
\begin{equation}\label{comp}
\underline{T}_{r}^{\ell}(\varepsilon) \leq \liminf_{t\to\infty} T_{r}^{\ell}(t)\leq\limsup_{t\to\infty} T_{r}^{\ell}(t)\leq \overline{T}_{r}^{\ell}(\varepsilon),\qquad \ell=1,\dots,p.
\end{equation}
Solutions of  limit equation \eqref{16} converge to a globally asymptotically stable equilibrium  by Proposition~\ref{pro21}, and by letting $\varepsilon\to 0$ we find that  this limit is the  same  as that of \eqref{15}.

As we have assumed that for all larger indices, the limits of the compartments  of the original system \eqref{eq-n-1} are equal to the limits obtained during the procedure, using  the equations for  $T_r^1(t),\dots,T_r^p(t)$ after $n-r+1$ cycles of the procedure satisfy the quasimonotone condition and the comparison \eqref{comp}, the limits obtained for these have to coincide with those of the original system (for $r=n$, the statement follows directly).

To prove that not only attractivity, but also global asymptotic stability holds
, we will again use induction. Let $E=(\bar{S}^1,\bar{T}_1^1,\dots,\bar{T}_n^1,\dots,\bar{S}^p,\bar{T}_1^p,\dots,\bar{T}_n^p)$ denote the  equilibrium  obtained at the end of the procedure,  where $\bar{T}_i^j=0$ or \mbox{$\bar{T}_i^j>0$}  depending on the stability moduli $(s(M_1),s(M_2),\ldots,s(M_n))$ and let $E_k=(\bar{S}^1,\bar{T}_1^1,\dots,\bar{T}_k^1,\dots,\bar{S}^p,\bar{T}_1^p,\dots,\bar{T}_k^p)$ be the equilibrium of the $p(k+1)$-dimensional system obtained during the procedure, consisting of the first $p(k+1)$ coordinates of~$E$. Let us suppose that $E_k$ is a stable equilibrium of the $p(k+1)$-dimensional reduced system for some $k\leq n$.
We will show that in each step, $E_{k+1}$  is a stable equilibrium of the $p(k+2)$-dimensional reduced system. Suppose  this does not hold, i.e.\ $E_{k+1}$ is unstable. In this case there exists an $\ep>0$ and is a sequence $\{x_m\}\to E_{k+1}$, $|x_m-E_{k+1}|<1/m$ such that the orbits started from the points of the sequence leave $B(E_{k+1},\ep)\coloneqq \{\,x\in\mathbb{R}_+^{(k+2)p} : |x-E_{k+1}|\leq \ep\,\}$. Let us denote by $x_m^{\ep}$ the first exit point from $B(E_{k+1},\ep)$ of the solution started from $x_m$, reached at time $\tau_m$. There is a convergent subsequence of the sequence $x_m^{\ep}$ (still denoted by $x_m^{\ep}$) which tends to a point denoted by $x_{\ep}^*\in S(E_{k+1},\ep)\coloneqq \{\,x\in\mathbb{R}_+^{(k+2)p} : |x-E_{k+1}|= \ep\,\}$. 
We will show that the  $E_{k+1}\in\alpha(x_{\ep}^*)$. For this end, let us consider the set $S(E_{k+1},\frac{\ep}{2})$. 
Clearly, all solutions started from the points $x_m$ (we drop the first elements of the sequence, if necessary) will leave the set $B(E_{k+1},\frac{\ep}{2})$. 
We denote the last exit point of each trajectory from this set  before time $\tau_m$, respectively, by $x_m^{\scriptscriptstyle{\ep/2}}$. Also this sequence has a convergent subsequence (still denoted the same way), let us denote its limit by  $x_{\scriptscriptstyle{\ep/2}}^*$. 
We will show that the trajectory started from $x_{\scriptscriptstyle{\ep/2}}^*$ goes through $x_\varepsilon^*$.  As $E_{k+1}$ is globally attractive,  this trajectory will eventually enter $S(E_{k+1},\frac{\ep}{4})$ at some time $T>0$.  Let us suppose that the trajectory started from $x_{\scriptscriptstyle{\ep/2}}^*$ does not go through $x_\varepsilon^*$ and let us denote by $d>0$ the distance of this trajectory from $x_\varepsilon^*$.
For continuity reasons, there is an $N\in\mathbb{N}$ so that for any  $m>N$, $|x_{\scriptscriptstyle{\ep/2}}^*t-x_m^{\scriptscriptstyle{\ep/2}}t|<\max\{\frac{d}{2},\frac{\ep}{8}\}$ for $0<t<T$. This means that for $m$ large enough, the trajectory started from $x_m^{\scriptscriptstyle{\ep/2}}$ will enter again  $S(E_{k+1},\frac{\ep}{2})$ without getting close to $x_\varepsilon*$ which contradicts either $x_m^\varepsilon$ being the first exit point from $B(E_{k+1},\varepsilon)$ or $x_m^{\scriptscriptstyle{\ep/2}}$ being the last exit point before $\tau_m$ from $B(E_{k+1},\frac{\varepsilon}{2})$. Hence, we have shown that the trajectory started from $x_{\scriptscriptstyle{\ep/2}}^*$ goes through $x_\varepsilon^*$. 
Proceeding like this (taking neighbourhoods of radius $\ep/4,\ \ep/8$ etc.) we obtain that the backward trajectory of $x_{{\ep}}^*$  enters any small neighbourhood of $E_{k+1}$ as $t\to-\infty$, hence, $E_{k+1}\in\alpha(x_\varepsilon^*)$, while it follows from the global attractivity of $E_{k+1}$ that the $\omega$-limit set of the trajectory is  $\{E_{k+1}\,\}$.
Let us denote this trajectory by $\gamma(x_\varepsilon^*)$

We know that the equations for $T_{k+1}^1(t),\dots,T_{k+1}^p(t)$ and $N_{k+1}^1(t),\dots,N_{k+1}^p(t)$ can be decoupled from the rest of the equations and using the exponential stability of the limits \eqref{limit1} and Proposition \ref{pro21} we obtain that $\bar T_{k+1}^1,\dots,\bar T_{k+1}^p$ is a stable equilibrium of the system consisting of the equations for  $\frac{\mathrm{d}}{\mathrm{d}t}{T_{k+1}^1}(t),\dots,\frac{\mathrm{d}}{\mathrm{d}t}{T_{k+1}^p}(t)$. 
Therefore, the equilibrium $E_{k+1}$ is stable in the  coordinates $T_{k+1}^1,\dots,T_{k+1}^p$ in the sense that for any $\tilde{\varepsilon}>0$ there exists a $\tilde{\delta}(\tilde{\varepsilon})>0$ such that for any initial value $x$ with $|x-E_{k+1}|<\tilde{\delta}$, $|T_{k+1}^\ell(t)-\bar T_{k+1}^\ell|<\tilde{\varepsilon}$ for all $t>0$ and $\ell=1,\dots,p$.
 Thus, the  trajectory $\gamma(x_\varepsilon^*)$ obtained above lies entirely  in the subspace $\{T_{k+1}^1=\bar{T}_{k+1}^1,\dots,T_{k+1}^p=\bar{T}_{k+1}^p\}$. On the other hand,  the current $p(k+2)$-dimensional system coincides with the $p(k+1)$-dimensional system on this subspace. For the latter system,  stability of  the equilibrium $E_k$ follows from the induction assumption. However, the existence of an orbit whose $\omega$-limit set is $\{E_{k+1}\}$   and whose $\alpha$-limit set contains $E_{k+1}$ contradicts the stability of the equilibrium $E_k$. This  implies the global asymptotic stability of the equilibrium of the $p(k+2)$-dimensional system. 

 For $k=1$, the assertion holds trivially, hence, repeating the inductive step we obtain   global asymptotic stability of the equilibrium $E$. 
 \end{proof}
 
 \section{Corrigendum of Theorem 2.2 of [\textit{Math Biosci Eng.} 2017;{14}:421--435]}
 In this section, we consider the special case of one patch examined in D\'enes, Muroya and R\"ost \cite{DMR} and give  a correction to the proof of Theorem 2.2 of that paper.
 First, we recall this theorem  about the globally asymptotically stable equilibrium of the multistrain SIS model
 \begin{equation}\label{regi}
 	\begin{aligned}
 		\frac{\mathrm{d}S(t)}{\mathrm{d}t}={}&B-b S(t)- S(t) \sum_{k=1}^n \beta_{kk} T_k(t)+\sum_{k=1}^n \theta_k T_k(t), \\
 		\frac{\mathrm{d}T_k(t)}{\mathrm{d}t}={}&S(t) \beta_{kk} T_k(t)+T_k(t) \sum_{j=1}^n (1-\delta_{kj}) \beta_{kj} T_j(t)-(b+\theta_k+\delta_{kn}d_n) T_k(t), \qquad k=1,2,\dots,n,
 	\end{aligned}
 \end{equation}
 with initial conditions
 \begin{equation*}
 	\begin{split}
 		& S\left( 0\right) =\phi_0, \qquad T_k\left( 0\right) =\phi_k, \qquad k=1,2,\dots, n, \\
 		&\left( \phi_0,\phi_1,\phi_2,\dots,\phi_n \right) \in \Gamma,
 	\end{split}
 \end{equation*}
 where $\delta_{kj}$ denotes the Kronecker delta such that $\delta_{kj}=1$ if $k=j$ and $\delta_{kj}=0$ otherwise, and $\Gamma=[0,\infty)^{n+1}$.
 We assume that the conditions
 \begin{equation*}
 	\begin{aligned}
 		\beta_{kj}&=\beta_{kk}, &\quad 1 &\leq j \leq k, \\
 		\beta_{kj}&=-\beta_{jk}=-\beta_{jj}, &\quad k+1 &\leq j \leq n,
 	\end{aligned}
 \end{equation*}
 hold for the infection rates for $k=1,2,\ldots,n$, i.e.\ we assume that the $k$-th strain infects those who are infected
 by a milder strain (including the non-infected) with the same rate. The notation $d_n$ stands for disease-induced death rate for the most infectious strain. 

In our previous work \cite{DMR}, we gave an iterative procedure (similar to the one introduced in Section 3 of the present paper) to calculate a sequence of reproduction numbers which completely determines the global dynamics of the system. In the  general step of the procedure we consider the system
\begin{equation}\label{lsteps-1}
\begin{aligned}
\frac{\mathrm{d}S(t)}{\mathrm{d}t}={}&B^{(\ell)}-b^{(\ell)} S(t)- S(t) \sum_{k=1}^{n-\ell} \beta_{kk} T_k(t)+\sum_{k=1}^{n-\ell} \theta_k T_k(t), \\
\frac{\mathrm{d}T_{k}(t)}{\mathrm{d}t}={}& S(t) \beta_{kk} T_k(t)+T_k(t) \sum_{j=1}^{n-\ell} (1-\delta_{kj}) \beta_{kj} T_j(t) -\left(b^{(\ell)}+{\theta}_k\right) T_k(t), \qquad k=1,2,\dots,n-\ell-1, \\
\end{aligned}
\end{equation}
and
\begin{equation}\label{lsteps-2}
\begin{aligned}
\frac{\mathrm{d}T_{n-\ell}(t)}{\mathrm{d}t}={}&S(t) \beta_{n-\ell,n-\ell} T_{n-\ell}(t)+T_{n-\ell}(t)\sum_{j=1}^{n-\ell} (1-\delta_{n-\ell,j}) \beta_{n-\ell,j} T_j(t)-\left(b^{(\ell)}+{\theta}_{n-\ell}\right) T_{n-\ell}(t), \\
\frac{\mathrm{d}N_{n-\ell}(t)}{\mathrm{d}t}={}&B^{(\ell)}-b^{(\ell)} N_{n-\ell}(t),
\end{aligned}
\end{equation}
where
$$N_{n-\ell}(t)=S(t)+\sum_{k=1}^{n-\ell}T_k(t),$$
$B^{(0)}=B, \ b^{(0)}=b$ and 
we define  $$\mathcal{R}_0^{(n-\ell)}\coloneqq \frac{B^{(\ell)} \beta_{n-\ell,n-\ell}}{b^{(\ell)}(b^{(\ell)}+\theta_{n-\ell})}$$ and
\begin{equation*}
B^{(\ell)}\coloneqq B^{(\ell-1)}+\theta_{n-\ell+1} T_{n-\ell+1}^*, \qquad b^{(\ell)}\coloneqq  b^{(\ell-1)}+\beta_{n-\ell+1,n-\ell+1}T_{n-\ell+1}^*,
\end{equation*}
if $\mathcal{R}_0^{(n-\ell)}>1$ and
\begin{equation*}
B^{(\ell)}\coloneqq B^{(\ell-1)}, \qquad b^{(\ell)}\coloneqq b^{(\ell-1)},
\end{equation*}
if $\mathcal{R}_0^{(n-\ell)} \leq 1$.

Now we introduce $U_{n-\ell}(t)=B^{(\ell)}/b^{(\ell)}-N_{n-\ell}(t)$, to rewrite the equation \eqref{lsteps-2} as
\begin{equation}\label{red_n-l}
\begin{aligned}
\frac{\mathrm{d}T_{n-\ell}(t)}{\mathrm{d}t}={}&\beta_{n-\ell,n-\ell}T_{n-\ell}(t)\left(\frac{B^{(\ell)}}{b^{(\ell)}}-\frac{b^{(\ell)}+{\theta}_{n-\ell}}{\beta_{n-\ell,n-\ell}}-T_ {n-\ell}(t)-U_{n-\ell}(t)\right), \\
\frac{\mathrm{d}U_{n-\ell}(t)}{\mathrm{d}t}={}&-b^{(\ell)} \ U_{n-\ell}(t).
\end{aligned}
\end{equation}
Again, \eqref{red_n-l} might be decoupled from the other equations \eqref{lsteps-1}.
For $\mathcal{R}_0^{(n-\ell)} \leq 1$, system \eqref{red_n-l} has only the trivial equilibrium $(0,0)$. But for $\mathcal{R}_0^{(n-\ell)}>1$, system \eqref{red_n-l} has two equilibria: the trivial equilibrium $(0,0)$ and the non-trivial equilibrium
$$(T_{n-\ell}^*,U_{n-\ell}^*)=\left(\frac{B^{(\ell)}\beta_{n-\ell,n-\ell}-(b^{(\ell)}+\theta_{n-\ell})b^{(\ell)}}{\beta_{n-\ell,n-\ell}b^{(\ell)}},0 \right),$$ which only exists if $$\mathcal{R}_0^{(n-\ell)}>1.$$

Then, from \eqref{lsteps-1}, we obtain the systems
\begin{equation*}\label{lsteps-3}
\begin{aligned}
\frac{\mathrm{d}S(t)}{\mathrm{d}t}={}&B^{(\ell+1)}-b^{(\ell+1)} S(t)- S(t) \sum_{k=1}^{n-\ell-1} \beta_{kk} T_k(t)+\sum_{k=1}^{n-\ell-1} \theta_k T_k(t), \\
\frac{\mathrm{d}T_{k}(t)}{\mathrm{d}t}={}& S(t) \beta_{kk} T_k(t)+T_k(t) \sum_{j=1}^{n-\ell-1} (1-\delta_{kj}) \beta_{kj} T_j(t) -\left(b^{(\ell+1)}+{\theta}_k\right) T_k(t), \qquad k=1,2,\dots,n-\ell-2, \\
\end{aligned}
\end{equation*}
and
\begin{equation*}\label{lsteps-4}
\begin{aligned}
\frac{\mathrm{d}T_{n-\ell-1}(t)}{\mathrm{d}t}={}&S(t) \beta_{n-\ell-1,n-\ell-1} T_{n-\ell-1}(t)+T_{n-\ell-1}(t)\sum_{j=1}^{n-\ell-1} (1-\delta_{n-\ell-1,j}) \beta_{n-\ell-1,j} T_j(t)-\left(b^{(\ell+1)}+{\theta}_{n-\ell-1}\right) T_{n-\ell-1}(t), \\
\frac{\mathrm{d}N_{n-\ell-1}(t)}{\mathrm{d}t}={}&B^{(\ell+1)}-b^{(\ell+1)} N_{n-\ell-1}(t),
\end{aligned}
\end{equation*}
where $$N_{n-\ell-1}(t)=S(t)+\sum_{k=1}^{n-\ell-1}T_k(t),$$
$B^{(0)}=B, \ b^{(0)}=b$ and 
 we define  $$\mathcal{R}_0^{(n-\ell-1)}\coloneqq \frac{B^{(\ell+1)} \beta_{n-\ell-1,n-\ell-1}}{b^{(\ell+1)}(b^{(\ell+1)}+\theta_{n-\ell-1})}$$ and
\begin{equation*}
B^{(\ell+1)}\coloneqq B^{(\ell)}+\theta_{n-\ell} T_{n-\ell}^*, \qquad b^{(\ell+1)}\coloneqq b^{(\ell)}+\beta_{n-\ell,n-\ell}T_{n-\ell}^*,
\end{equation*}
if $\mathcal{R}_0^{(n-\ell-1)}>1$ and
\begin{equation*}
B^{(\ell+1)}\coloneqq B^{(\ell)}, \qquad b^{(\ell+1)}\coloneqq b^{(\ell)},
\end{equation*}
if $\mathcal{R}_0^{(n-\ell-1)} \leq 1$, which, again, are systems with the same structure. In the end, we arrive at the two-dimensional system
\begin{equation*}
\begin{split}
\frac{\mathrm{d}S(t)}{\mathrm{d}t}={}& B^{(n-1)}-b^{(n-1)}S(t)-\beta_{11}S(t)T_1(t)+\theta_1T_1(t),\\
\frac{\mathrm{d}T_1(t)}{\mathrm{d}t}={}&\beta_{11}S(t)T_1(t)-(b^{(n-1)}+\theta_1)T_1(t),
\end{split}
\end{equation*}
which has the two equilibria $$\left(\frac{B^{(n-1)}}{b^{(n-1)}},0\right)\quad\mbox{and}\quad\left(\frac{b^{(n-1)}+\theta_1}{\beta_{11}},\frac{B^{(n-1)}}{b^{(n-1)}}-\frac{b^{(n-1)}+\theta_1}{\beta_{11}}\right),$$
with the latter one only existing if $$\mathcal{R}_0^{(n)}\coloneqq\frac{B^{(n-1)}\beta_{11}}{b^{(n-1)}(b^{(n-1)}+\theta_1)}>1.$$
 The dynamics of this system can be determined in a similar way as in the  case of \eqref{red_n-l}, and we obtain that the first equilibrium is globally asymptotically stable if $\mathcal{R}_0^{(n)}\leq 1$ and the second one is globally asymptotically stable if $\mathcal{R}_0^{(n)}> 1$.

 \begin{theorem}[Theorem 2.2 of D\'enes, Muroya, R\"ost~\cite{DMR}]
 The multistrain SIS model \eqref{regi} (equation $(1)$ in D\'enes, Muroya, R\"ost~\cite{DMR}) has a globally asymptotically stable
 equilibrium on the region $\Gamma_0$, where $\Gamma_0$ is the interior of $\ \Gamma$. The global dynamics
 is completely determined by the threshold parameters $\mathcal{R}_0^{(1)},\dots,\mathcal{R}_0^{(n)}$,
 which
 can be obtained iteratively and determine which one of the equilibria is globally
 asymptotically stable.
 \end{theorem}
 	 	\begin{proof}
 Let us suppose that there exists a solution started with positive initial values whose limit is not the equilibrium $E$ obtained at the end of the procedure described in D\'enes, Muroya, R\"ost~\cite{DMR}. It follows from the procedure that the last coordinate tends to the last coordinate of~$E$. There exists a maximal index $k$ ($0\leq k\leq n-1$) such that the $k$th coordinate of the solution does not tend to the $k$th coordinate of~$E$, while all coordinates with index larger than $k$ do tend to the corresponding coordinate of $E$.  Let us consider the $k$th equation in the original system:
 $$\frac{\mathrm{d} T_k}{\mathrm{d} t}=S(t)\beta_{kk}T_k(t)+T_k(t)\sum_{i=1}^{k-1}\beta_{kk}T_i(t)-T_k(t)\sum_{i=k+1}^{n}\beta_{ii}T_i(t)-(b+\theta_k)T_k(t).$$ 
 Introducing the notation  $T_0(t)\coloneqq S(t)$, let us define $\tilde N_k(t)$ as 
 $$\tilde N_k(t)\coloneqq \sum_{i=0}^{k}T_k(t)$$
 with respect to the original system.
 
 Hence, we can write the equation for $T_k(t)$ as
 \begin{align}
 \frac{\mathrm{d} T_k}{\mathrm{d} t}={}&\left(\tilde N_k(t)-\sum_{i=1}^{k}T_i(t)\right)\beta_{kk}T_k(t)+T_k(t)\sum_{i=1}^{k-1}\beta_{kk}T_i(t)-T_k(t)\sum_{i=k+1}^{n}\beta_{ii}T_i(t)-(b+\theta_k)T_k(t)\nonumber\\
 ={}&\tilde N_k(t)\beta_{kk}T_k(t)-\beta_{kk}(T_k(t))^2-T_k(t)\sum_{i=k+1}^{n}\beta_{ii}T_i(t)-(b+\theta_k)T_k(t)\label{00}
 \shortintertext{and} 
 \frac{\mathrm{d} \tilde{N}_k}{\mathrm{d} t}={}&B-b\tilde N_k(t)-\tilde{N}_k(t)\sum_{i=k+1}^n\beta_{ii}T_i(t)+\sum_{i=k+1}^{n}\theta_i T_{i}(t)\label{01}.
 \end{align}

 For an arbitrary small $\varepsilon>0$, there exists a $t_1>0$ such that if $t>t_1$, then for all $i>k$, $|T_i(t)-T_i^*|<\frac{\varepsilon}{\max\{\beta_{ii},\theta_i\}n}$. Hence, for the terms multiplied by $\tilde{N}_k(t)$ in  equation \eqref{01},  the following estimates hold for $t>t_1$:
 \begin{align*}
 b+\sum_{i=k+1}^n\beta_{ii}T_i^*-\varepsilon&\leq b+\sum_{i=k+1}^n\beta_{ii}T_i(t)\leq b+\sum_{i=k+1}^n\beta_{ii}T_i^*+\varepsilon,
 \end{align*}
 and for the rest of the terms the estimates
 \begin{align*}
 B+\sum_{i=k+1}^{n}\theta_i T_{i}^*-\varepsilon\leq B+\sum_{i=k+1}^{n}\theta_i T_{i}(t)\leq B+\sum_{i=k+1}^{n}\theta_i T_{i}^*+\varepsilon.
 \end{align*}
 From these, we can get the following estimation for $\frac{\mathrm{d} \tilde N_k}{\mathrm{d} t}$ for $t>t_1$:
 $$B+\sum_{i=k+1}^{n}\theta_i T_{i}^*-\varepsilon -\tilde N_k(t)\left(b+\sum_{i=k+1}^n\beta_{ii}T_i^*+\varepsilon\right)\le\frac{\mathrm{d} \tilde N_k}{\mathrm{d} t}\leq  B+\sum_{i=k+1}^{n}\theta_i T_{i}^*+\varepsilon -\tilde N_k(t) \left(b+\sum_{i=k+1}^n\beta_{ii}T_i^*-\varepsilon \right)$$
 
 Taking into consideration that $b+\sum_{i=k+1}^n\beta_{ii}T_i^*=b^{(n-k)}$ and $B+\sum_{i=k+1}^{n}\theta_i T_{i}^*=B^{(n-k)}$ (see~D\'enes, Muroya, R\"ost\cite{DMR}), using a comparison principle, one obtains that the limit of  equation \eqref{01} is the same as that of the corresponding system during the procedure, let us denote this limit by $\tilde N_k^*$.
 
 From the above estimations and equation \eqref{00}, we obtain that there exists a $t_2>0$ such that for all $t>t_2$ the following estimates can be given for $\frac{\mathrm{d} T_k}{\mathrm{d} t}$:
 \begin{align*}
 \frac{\mathrm{d} T_k}{\mathrm{d} t}&\leq\beta_{kk}(\tilde N_k^*+\varepsilon)T_k(t)-\beta_{kk}(T_k(t))^2-T_k(t)\left(\sum_{i=k+1}^{n}\beta_{ii}T_i^*-\varepsilon\right)-(b+\theta_k)T_k(t)
 \shortintertext{and}
 \frac{\mathrm{d} T_k}{\mathrm{d} t}&\geq\beta_{kk}(\tilde N_k^*-\varepsilon)T_k(t)-\beta_{kk}(T_k(t))^2-T_k(t)\left(\sum_{i=k+1}^{n}\beta_{ii}T_i^*+\varepsilon\right)-(b+\theta_k)T_k(t).
 \end{align*}
 
 Now, using a similar comparison argument as before, one can see that the limit of the solution of the equation for $T_k(t)$ is the same as that of the corresponding equation during the procedure, depending on the same reproduction number. 

 The rest of the proof (the proof of stability) remains the same as given in Theorem 2.2 of D\'enes, Muroya, R\"ost \cite{DMR}.
  	\end{proof}
 
\section*{Acknowledgements}
A.~D\'enes was supported by Hungarian Scientific Research Fund OTKA PD 112463 and National Research, Development and Innovation Office NKFIH KH 125628 and the J\'anos Bolyai Research Scholarship of the Hungarian Academy of Sciences.
Y.~Muroya was supported by Scientific Research (c), No.~24540219 of Japan Society for the Promotion of Science.
G.~R\"ost was supported by the EU-funded Hungarian grant EFOP-3.6.1-16-2016-00008 and Marie Sk\l odowska-Curie Grant No.~748193.

\end{document}